\documentclass[12pt,a4paper,twoside,american,english,british]{scrartcl}
\usepackage{helvet}
\usepackage[T1]{fontenc}
\usepackage[latin9]{inputenc}
\usepackage{amsmath}
\usepackage{amssymb}

\makeatletter

 \usepackage{amsthm}
 \theoremstyle{plain}
\newtheorem{thm}{Theorem}[section]
  \theoremstyle{definition}
  \newtheorem{defn}[thm]{Definition}
  \theoremstyle{remark}
  \newtheorem{rem}[thm]{Remark}

\usepackage[T1]{fontenc}    % Silbentrennung auch nach Umlauten mglich

\usepackage{a4wide}        % A4-Format
\usepackage{tu-preprint}
\usepackage{amsmath}

\usepackage{euscript}
\usepackage{mathabx}
\allowdisplaybreaks[1] % allow page breaks in some displayed equations
\usepackage{amsfonts}
\newenvironment{keywords}{ \noindent\footnotesize\textbf{Keywords and phrases:}}{}

\newenvironment{class}{\noindent\footnotesize\textbf{Mathematics subject classification 2000:}}{}

\usepackage{color,tu-preprint}

\renewcommand*{\i}{\mathrm{i}}
\DeclareMathAccent{\Circ}{\mathalpha}{operators}{"17}
\newcommand{\interior}[1]{\Circ{#1}}
\renewcommand{\Im}{\operatorname{\mathfrak{Im}}}
\renewcommand{\Re}{\operatorname{\mathfrak{Re}}}

\renewcommand*{\epsilon}{\varepsilon}

\makeatother

\usepackage{babel}
\begin{document}
\institut{Institut f\"ur Analysis}

\preprintnumber{MATH-AN-04-2012}

\preprinttitle{Evolutionary Problems Involving Sturm-Liouville Operators.}

\author{Rainer Picard \& Bruce Watson}

\makepreprinttitlepage

\setcounter{section}{-1}

%\date{}

\title{Evolutionary Problems Involving Sturm-Liouville Operators.}

\author{Rainer Picard\\
Institut für Analysis, Fachrichtung Mathematik\\
 Technische Universität Dresden\\
 Germany\\
 rainer.picard@tu-dresden.de\\
\&\\
Bruce Watson%
\thanks{funded in part by NRF grant IFR2011032400120%
} \\
School of Mathematics\\
 University of the Witwatersrand, Johannesburg\\
 South Africa\\
 bruce.watson@wits.ac.za }
\maketitle
\begin{abstract}
The purpose of this paper is to further exemplify an approach to evolutionary
problems originally developed in \cite{Pi2009-1}, \cite{2009-2}
for a special case and extended to more general evolutionary problems,
see \cite{PAMM:PAMM201110333}, compare the survey article \cite{PIC_2010:1889}.
The ideas presented in there are utilized for $\left(1+1\right)$-dimensional
evolutionary problem, which in a particular case results in a hyperbolic
partial differential equation with a Sturm-Liouville type spatial
operator constrained by a impedance type boundary condition. 
\end{abstract}

\begin{class}
34B24 Sturm-Liouville theory
35F10 Initial value problems for linear first-order PDE, linear evolution equations,
35A22 Transform methods (e.g. integral transforms),
47G20 Integro-differential operators,
34L40 Particular operators (Dirac, one-dimensional Schrödinger, etc.)
35K90 Abstract parabolic evolution equations,
35L90 Abstract hyperbolic evolution equations.
\end{class}

\section{Introduction}

A canonical form of many linear evolutionary%
\footnote{We prefer the term \emph{evolutionary} equations, since the term \emph{evolution
equations} is usually reserved for a special case of the class of
evolutionary equations considered here.%
} equations of mathematical physics is given by a dynamic system of
equations 
\begin{eqnarray*}
\partial_{0}V+AU & = & F
\end{eqnarray*}
completed by a so-called material law
\[
V=\mathcal{M}U.
\]
Here $\partial_{0}$ denotes derivation with respect to the time variable,
$A$ is a usually unbounded operator containing spatial derivatives
and $\mathcal{M}$ is a continuous linear operator. Here we would
like to inspect more closely a very specific situation, where the
space dimension is 1. Although this is a rather particular case it
has the advantage that an impedance type boundary condition, which
we wish to consider, can be considered in a more {}``tangible''
way without incurring regularity assumptions on coefficients and data.
In the higher dimensional case for the acoustic equations, which can
also be discussed with no further regularity requirements, the constraints
on the impedance type boundary condition are much less explicit, compare
\cite{PAMM:PAMM201110333}, \cite{PIC_2010:1889}. Moreover, we hope
to gain a different access to a class of problems, which are closely
linked to Sturm-Liouville operators, in fact yielding a generalization
of such operators. That we are discussing the direct time-dependent
problem rather than an associated spectral problem will actually be
advantageous, since it provides a simpler access to the discussion
of well-posedness. 

More specifically we want to consider 
\[
A=\left(\begin{array}{ccc}
0 & 0 & \partial\\
0 & 0 & 0\\
\partial & 0 & 0
\end{array}\right)
\]
as a differential operator on the unit interval $]-1/2,1/2[$ with
an impedance type boundary condition of the form 
\[
\partial_{0}\alpha\left(\pm1/2\mp0\right)s\left(\:\cdot\:,\pm1/2\mp0\right)-v\left(\:\cdot\:,\pm1/2\mp0\right)=0
\]
holding on the real time-line $\mathbb{R}$ as a constraint characterizing
$\left(\begin{array}{c}
s\\
w\\
v
\end{array}\right)$ in the domain $D\left(A\right).$ Here $\partial$ denotes the spatial
derivative and $\alpha$ is a coefficient operator specified more
precisely later. We shall focus here on the time-translation invariant,
i.e. autonomous, case. This means that time-translation and consequently
time-differentiation commutes with $\alpha$, $\mathcal{M}$ and $A.$

Our discussion is embedded into an abstract setting, which we will
develop in Section \ref{sec:Space-Time-Evolution-Equations} first.
In Section \ref{sec:Application:-Acoustic} we will then discuss our
problem of interest as an application of the solution theory in the
abstract setting.

\section{The Abstract Solution Framework\label{sec:Space-Time-Evolution-Equations}}

\subsection{Sobolev chains associated with the time-derivative}

A particular instance of the construction of Sobolev chains is the
one based on the time-derivative $\partial_{0}$. We recall, e.g.
from \cite{2009-2,PIC_2010:1889}, that differentiation considered
in the complex Hilbert space $H_{\nu,0}(\mathbb{R}):=\{f\in L_{\textnormal{loc}}^{2}(\mathbb{R})|(x\mapsto\exp(-\nu x)f(x))\in L^{2}(\mathbb{R})\}$,
$\nu\in\mathbb{R}\setminus\left\{ 0\right\} $, with inner product
\[
(f,g)\mapsto\langle f,g\rangle_{\nu,0}:=\int_{\mathbb{R}}f(x)^{*}g(x)\:\exp(-2\nu x)\: dx
\]
can indeed be established as a normal operator, which we denote by
$\partial_{0,\nu}$, with 
\[
\Re\partial_{0,\nu}=\nu.
\]
For $\Im\partial_{0,\nu}$ we have as a spectral representation the
Fourier-Laplace transform $\mathcal{L}_{\nu}:H_{\nu,0}(\mathbb{R})\to L^{2}\left(\mathbb{R}\right)$
given by the unitary extension of 
\begin{align*}
\interior C_{\infty}\left(\mathbb{R}\right)\subseteq H_{\nu,0}(\mathbb{R}) & \to L^{2}(\mathbb{R})\\
\phi & \mapsto\left(x\mapsto\frac{1}{\sqrt{2\pi}}\int_{\mathbb{R}}\exp\left(-\mathrm{i}xy\right)\;\exp\left(-\nu y\right)\phi(y)\; dy\right).
\end{align*}
 In other words, we have the unitary equivalence
\[
\Im\partial_{0,\nu}=\mathcal{L}_{\nu}^{-1}m\:\mathcal{L}_{\nu},
\]
where $m$ denotes the selfadjoint multiplication-by-argument operator
in $L^{2}\left(\mathbb{R}\right)$. Since $0$ is in the resolvent
set of $\partial_{0,\nu}$ we have that $\partial_{0,\nu}^{-1}$ is
an element of the Banach space $L\left(H_{\nu,0}(\mathbb{R}),H_{\nu,0}(\mathbb{R})\right)$
of continuous (left-total) linear mappings in $H_{\nu,0}(\mathbb{R})$.
Denoting generally the operator norm of the Banach space $L\left(X,Y\right)$
by $\left\Vert \:\cdot\;\right\Vert _{L\left(X,Y\right)}$, we get
for $\partial_{0}^{-1}$
\[
\left\Vert \partial_{0,\nu}^{-1}\right\Vert _{L\left(H_{\nu,0}(\mathbb{R}),H_{\nu,0}(\mathbb{R})\right)}=\frac{1}{\left|\nu\right|}.
\]
Not too surprisingly, we find for $\nu>0$ 
\[
\left(\partial_{0,\nu}^{-1}\varphi\right)\left(x\right)=\int_{-\infty}^{x}\varphi\left(t\right)\: dt
\]
and for $\nu<0$
\[
\left(\partial_{0,\nu}^{-1}\varphi\right)\left(x\right)=-\int_{x}^{\infty}\varphi\left(t\right)\: dt
\]
for all $\varphi\in H_{\nu,0}\left(\mathbb{R}\right)$ and $x\in\mathbb{R}$.
Since we are interested in the forward causal situation, we assume
$\nu>0$ throughout. Moreover, in the following we shall mostly write
$\partial_{0}$ for $\partial_{0,\nu}$ if the choice of $\nu$ is
clear from the context. 

Thus, we obtain a chain $\left(H_{\nu,k}\left(\mathbb{R}\right)\right)_{k\in\mathbb{Z}}$
of Hilbert spaces, where$H_{\nu,k}\left(\mathbb{R}\right)$ is the
completion of the inner product space $D\left(\partial_{0}^{k}\right)$
with norm $\left|\:\cdot\:\right|_{\nu,k}$ given by
\[
\phi\mapsto\left|\partial_{0}^{k}\phi\right|_{\nu,0}.
\]
Similarly, for $\i m+\nu$ as a normal operator in $L^{2}\left(\mathbb{R}\right)$
we construct the chain of polynomially weighted $L^{2}\left(\mathbb{R}\right)$-spaces
\[
\left(L_{k}^{2}\left(\mathbb{R}\right)\right)_{k\in\mathbb{Z}}
\]
with
\[
L_{k}^{2}\left(\mathbb{R}\right)\::=\left\{ f\in L_{\mathrm{loc}}^{2}\left(\mathbb{R}\right)|\,\left(\i m+\nu\right)^{k}f\in L^{2}\left(\mathbb{R}\right)\right\} =H_{k}\left(\i m+\nu\right)
\]
for $k\in\mathbb{Z}$.

Since the unitarily equivalent operators $\partial_{0,\nu}$ and $\i m+\nu$
(via the Fourier-Laplace transform) can canonically be lifted to the
$X$-valued case, $X$ an arbitrary complex Hilbert space, we are
lead to a corresponding chain $\left(H_{\nu,k}\left(\mathbb{R},X\right)\right)_{k\in\mathbb{Z}}$
and $\left(L_{k}^{2}\left(\mathbb{R},X\right)\right)_{k\in\mathbb{Z}}$
of $X$-valued generalized functions. The Fourier-Laplace transform
can also be lifted to the $X$-valued case yielding
\begin{eqnarray*}
H_{\nu,k}\left(\mathbb{R},X\right) & \to & L_{k}^{2}\left(\mathbb{R},X\right)\\
f & \mapsto & \mathcal{L}_{\nu}f
\end{eqnarray*}
as a unitary mapping for $k\in\mathbb{N}$ and by continuous extension,
keeping the notation $\mathcal{L}_{\nu}$ for the extension, also
for $k\in\mathbb{Z}$. Since $\mathcal{L}_{\nu}$ has been constructed
from a spectral representation of $\Im\partial_{0,\nu}$, we can utilize
the corresponding operator function calculus for functions of $\Im\partial_{0,\nu}$.
Noting that $\partial_{0}=\i\Im\partial_{0,\nu}+\nu$ is a function
of $\Im\partial_{0,\nu}$ we can define operator-valued functions
of $\partial_{0}$. 
\begin{defn}
Let $r>\frac{1}{2\nu}>0$ and $M:B_{\mathbb{C}}(r,r)\to L(H,H)$ be
bounded and analytic, $H$ a Hilbert space. Then define 
\[
M\left(\partial_{0}^{-1}\right):=\mathcal{L}_{\nu}^{*}\: M\left(\frac{1}{\mathrm{i}m+\nu}\right)\:\mathcal{L}_{\nu},
\]
 where 
\[
M\left(\frac{1}{\mathrm{i}m+\nu}\right)\phi(t):=M\left(\frac{1}{\mathrm{i}t+\nu}\right)\phi(t)\quad(t\in\mathbb{R})
\]
for $\phi\in\interior C_{\infty}\left(\mathbb{R},H\right)$. \end{defn}
\begin{rem}
\label{Rem: rho-Independent} The definition of $M(\partial_{0}^{-1})$
is largely independent of the choice of $\nu$ in the sense that the
operators for two different parameters $\nu_{1},\nu_{2}$ coincide
on the intersection of the respective domains. 
\end{rem}
Simple examples are polynomials in $\partial_{0}^{-1}$ with operator
coefficients. A more exotic example of an analytic and bounded function
of $\partial_{0}^{-1}$ is the delay operator, which itself is a special
case of the time translation: 

\textbf{Examples}: Let $r>0$, $\nu>\frac{1}{2r}$, $h\in\mathbb{R}$
and $u\in H_{\nu,0}(\mathbb{R},X)$. We define 
\[
\tau_{h}u:=u(\:\cdot\:+h).
\]
 The operator $\tau_{h}\in L(H_{\nu,0}(\mathbb{R},X),H_{\nu,0}(\mathbb{R},X))$
is called a \emph{time-translation operator}. If $h<0$ the operator
$\tau_{h}$ is also called a \emph{delay operator}. In the latter
case the function 
\[
B_{\mathbb{C}}(r,r)\ni z\mapsto M(z):=\exp(z^{-1}h)
\]
 is analytic and uniformly bounded for every $r\in\mathbb{R}_{>0}$
(considered as an $L\left(X,X\right)$-valued function). An easy computation
shows for $u\in H_{\nu,0}\left(\mathbb{R},X\right)$ that 
\[
u(\:\cdot\:+h)=\mathcal{L}_{\nu}^{*}\exp((\mathrm{i}m+\nu)h)\mathcal{L}_{\nu}u=M(\partial_{0}^{-1})u=\exp(\left(\partial_{0}^{-1}\right)^{-1}h)\: u.
\]
This shows that
\[
\tau_{h}=\exp\left(h/\partial_{0}^{-1}\right)=\exp\left(h\partial_{0}\right).
\]

Another class of interesting bounded analytic functions of $\partial_{0}^{-1}$
are mappings produced by a temporal convolution with a suitable operator-valued
integral kernel.

\subsection{Abstract Solution Theory}

We shall discuss equations of the form
\begin{equation}
\left(\partial_{0}M\left(\partial_{0}^{-1}\right)+A\right)U=\mathcal{J}.\label{eq:evo-abstract}
\end{equation}
where we shall assume that $A$ and $A^{*}$ are commuting with $\partial_{0}$
and non-negative in the Hilbert space $H_{\nu,0}\left(\mathbb{R},H\right)$,
$H$ a given Hilbert space, in the sense that
\[
\Re\left\langle U|AU\right\rangle _{\nu,0}\geq0,\:\Re\left\langle V|A^{*}V\right\rangle _{\nu,0}\geq0
\]
for all $U\in D\left(A\right)$, $V\in D\left(A^{*}\right),$ and
$M$ is a material law in the sense of \cite{Pi2009-1,PDE_DeGruyter}.
More specifically we assume that $M$ is of the form 
\[
M\left(z\right)=M_{0}+zM_{1}+z^{2}M^{\left(2\right)}\left(z\right)
\]
where $M^{\left(2\right)}$ is an analytic and bounded $L\left(H,H\right)$-valued
function in a ball $B_{\mathbb{C}}\left(r,r\right)$ for some $r\in\mathbb{R}_{>0}$
and $M_{0}$ is a continuous, selfadjoint and non-negative operator
in $H$. The operator $M_{1}\in L\left(H,H\right)$ is such that
\begin{equation}
\nu M_{0}+\Re M_{1}\geq c_{0}>0\label{eq:posdefmat}
\end{equation}
for all sufficiently large $\nu\in\mathbb{R}_{>0}.$ The operator
$M\left(\partial_{0}^{-1}\right)$ is then to be understood in the
sense of the operator-valued function calculus associated with the
selfadjoint operator $\Im\left(\partial_{0}\right)=\frac{1}{2\i}\left(\partial_{0}-\partial_{0}^{*}\right)$. 

The appropriate setting turns out to be the Sobolev chain 
\[
\left(H_{\nu,k}\left(\mathbb{R},\: H\right)\right)_{k,s\in\mathbb{Z}}.
\]
From \cite{Pi2009-1,PDE_DeGruyter} we paraphrase the following solution
result.
\begin{thm}
\label{SolutionTheory}For $\mathcal{J}\in H_{\nu,k}\left(\mathbb{R},\: H\right)$
the problem (\ref{eq:evo-abstract}) has a unique solution $U\in H_{\nu,k}\left(\mathbb{R},\: H\right)$.
Moreover, 
\[
F\mapsto\left(\partial_{0}M\left(\partial_{0}^{-1}\right)+A\right)^{-1}F
\]
is a linear mapping in $L\left(H_{\nu,k}\left(\mathbb{R},\: H\right),H_{\nu,k}\left(\mathbb{R},\: H\right)\right),$
$k\in\mathbb{Z}.$ These mappings are causal in the sense that if
$F\in H_{\nu,k}\left(\mathbb{R},\: H\right)$ vanishes on the time
interval $]-\infty,\, a]$, then so does \textup{$\left(\partial_{0}M\left(\partial_{0}^{-1}\right)+A\right)^{-1}F$,
$a\in\mathbb{R}$, $k\in\mathbb{Z}$. }\end{thm}
\begin{rem}
Note that if $U\in H_{\nu,k}\left(\mathbb{R},\: H\right)$ and $\mathcal{J}\in H_{\nu,k}\left(\mathbb{R},\: H\right)$
equation (\ref{eq:evo-abstract}) actually makes sense in $H_{\nu,k-1}\left(\mathbb{R},\: H\right).$
Initially the solution theory is for the closure of $\left(\partial_{0}M\left(\partial_{0}^{-1}\right)+A\right)$
as a closed operator in $H_{\nu,k-1}\left(\mathbb{R},\: H\right),$
but it is
\[
\left(\partial_{0}M\left(\partial_{0}^{-1}\right)+A\right)U=\overline{\left(\partial_{0}M\left(\partial_{0}^{-1}\right)+A\right)}\, U
\]
with equality holding in $H_{\nu,k-1}\left(\mathbb{R},\: H\right).$
Indeed, for $\phi\in H_{\nu,k}\left(\mathbb{R},H\right)\cap D$$\left(A^{*}\right)$
we have
\begin{eqnarray*}
\left\langle \phi|\overline{\left(\partial_{0}M\left(\partial_{0}^{-1}\right)+A\right)}\, U\right\rangle _{\nu,k-1,0} & = & \left\langle \left(\partial_{0}^{*}M\left(\partial_{0}^{-1}\right)^{*}+A^{*}\right)\phi|U\right\rangle _{\nu,k-1,0}\\
 & = & \left\langle \phi|M\left(\partial_{0}^{-1}\right)\partial_{0}U\right\rangle _{\nu,k-1,0}+\left\langle A^{*}\phi|U\right\rangle _{\nu,k-1,0}
\end{eqnarray*}
and we read off that $U\in D\left(A\right)$ if $A$ is considered
in $H_{\nu,k-1}\left(\mathbb{R},H\right)$ (rather than $H_{\nu,k}\left(\mathbb{R},H\right)$)
giving 
\begin{eqnarray*}
AU & = & \overline{\left(\partial_{0}M\left(\partial_{0}^{-1}\right)+A\right)}\, U-M\left(\partial_{0}^{-1}\right)\,\partial_{0}U,\\
 & = & \overline{\left(\partial_{0}M\left(\partial_{0}^{-1}\right)+A\right)}\, U-\partial_{0}M\left(\partial_{0}^{-1}\right)\, U.
\end{eqnarray*}
The rigorous argument is somewhat more involved, see \cite{PIC_2010:1889,2009-2}.
This observation, however, motivates dropping the closure bar throughout.
\end{rem}

\section[\selectlanguage{american}%
Sturm-Liouville Evolutionary Problems\foreignlanguage{english}{}\selectlanguage{british}
]{\label{sec:Application:-Acoustic}An Application: An Evolutionary
Problem Involving a Sturm-Liouville Type Operator with an Impedance
Type Boundary Condition}

We exemplify the outlined theory by a $\left(1+1\right)$-dimensional
example. 

Consider%
\footnote{This $3\times3$-system has been chosen, rather than an alternative
$2\times2$-system formulation, since it shows conservativity in a
more obvious way, see Footnote \ref{fn:A-more-common}.%
} 
\[
A=\left(\begin{array}{ccc}
0 & 0 & \partial\\
0 & 0 & 0\\
\partial & 0 & 0
\end{array}\right)
\]
with an impedance type boundary condition implemented in the domain
of $A$ given by{\small 
\[
\left\{ \left(\begin{array}{c}
s\\
w\\
v
\end{array}\right)\in H_{\nu,0}\left(\mathbb{R},H\left(\partial,I\right)\oplus L^{2}\left(I\right)\oplus H\left(\partial,I\right)\right)\Big|\: a\left(\partial_{0}^{-1}\right)s-\partial_{0}^{-1}v\in H_{\nu,0}\left(\mathbb{R},H\left(\interior{\partial},I\right)\right)\right\} ,
\]
}where $H\left(\interior{\partial},I\right)$ denotes the completion
of the space of smooth function with compact support in $I=]-1/2,1/2[$
with respect to the graph norm of the derivative operator $\partial$.
The space $H\left(\partial,I\right)$ is the domain of the adjoint
of $\interior{\partial}$ also equipped with the corresponding graph
norm. We focus here on the finite interval case. It should be noted,
however, that the same reasoning would likewise work for the half
infinite interval case, say $I=\mathbb{R}_{>0}$. Indeed this case
would be in a sense simpler since only one boundary point would need
to be considered. We assume
\[
x\mapsto a\left(x,\:\cdot\:\right)
\]
to be uniformly continuous bounded-analytic-function-valued mapping.
As a matter of simplification we shall assume that $\left(z\mapsto a\left(x,z\right)\right)_{x\in I}$
is a uniformly bounded family of bounded functions, which are analytic
in a ball $B_{\mathbb{C}}\left(0,2r\right)$ centred at $0$ with
radius $r\in\mathbb{R}_{>0}$. Then surely $\left(z\mapsto a\left(x,z\right)\right)_{x\in I}$
is also a bounded family of analytic functions in $B_{\mathbb{C}}\left(r,r\right)$
and for $\nu>\frac{1}{2r}$ we have a continuous linear and causal
mapping: 
\begin{eqnarray*}
a\left(\partial_{0}^{-1}\right):H_{\nu,0}\left(\mathbb{R},L^{2}\left(I\right)\right) & \to & H_{\nu,0}\left(\mathbb{R},L^{2}\left(I\right)\right)\\
\varphi & \mapsto & \left(t\mapsto\left(x\mapsto a\left(x,\partial_{0}^{-1}\right)\varphi\left(t,x\right)\right)\right).
\end{eqnarray*}
Assuming further that the distributional derivative $a^{\prime}\coloneq\left(x\mapsto a\left(x,\:\cdot\:\right)\right)^{\prime}$
is such that 
\[
\left(z\mapsto a^{\prime}\left(x,z\right)\right)_{x\in I\setminus N}
\]
is also a uniformly bounded family of bounded functions, which are
analytic in $B_{\mathbb{C}}\left(0,2r\right)$, $N$ a Lebesgue null
set, we get a bounded linear and causal mapping 
\[
a^{\prime}\left(\partial_{0}^{-1}\right):H_{\nu,0}\left(\mathbb{R},L^{2}\left(I\right)\right)\to H_{\nu,0}\left(\mathbb{R},L^{2}\left(I\right)\right)
\]
and the product rule 
\[
\partial\left(a\left(\partial_{0}^{-1}\right)\, s\right)=a^{\prime}\left(\partial_{0}^{-1}\right)\: s+a\left(\partial_{0}^{-1}\right)\partial s
\]
holds for $s\in D\left(\partial\right)$. As another simplification
we assume that $a$ is real in the sense that
\[
a\left(x,z\right)^{*}=a\left(x,z^{*}\right)
\]
for $x\in I$ and $z\in B_{\mathbb{C}}\left(r,r\right)$.

Such an operator $A$ combined with a suitable material law yields
an evolutionary problem of the form
\begin{equation}
\left(\partial_{0}M\left(\partial_{0}^{-1}\right)+A\right)U=\mathcal{J}.\label{eq:Sturm-Louis}
\end{equation}
 Indeed, we consider a material law operators of the form%
\footnote{\label{fn:A-more-common}This implies a polynomial type material law
operator of the form 
\[
M\left(\partial_{0}^{-1}\right)=\left(\begin{array}{ccc}
\kappa_{0} & 0 & 0\\
0 & \kappa_{1} & 0\\
0 & 0 & \epsilon
\end{array}\right)+\partial_{0}^{-1}\left(\begin{array}{ccc}
0 & 0 & 0\\
0 & 0 & -\mu_{0}^{*}\\
0 & \mu_{0} & \eta
\end{array}\right)+\partial_{0}^{-2}\left(\begin{array}{ccc}
0 & 0 & 0\\
0 & 0 & 0\\
0 & 0 & \mu_{1}
\end{array}\right).
\]
We see that for $\kappa_{0},\,\kappa_{1},\,\epsilon$ strictly positive
continuous linear operators, $\mu_{1}=0$ and $\eta=0$ and $A$ skew-selfadjoint,
e.g. if $a\left(\partial_{0}^{-1}\right)=0$, we have a conservative
system since then $M^{\left(2\right)}\left(\partial_{0}^{-1}\right)=0$
and 
\[
M_{1}=\left(\begin{array}{ccc}
0 & 0 & 0\\
0 & 0 & -\mu_{0}^{*}\\
0 & \mu_{0} & \eta
\end{array}\right)
\]
is skew-selfadjoint making $A+M_{1}$ and so also $\sqrt{M_{0}^{-1}}\left(A+M_{1}\right)\sqrt{M_{0}^{-1}}$
skew-selfadjoint. Consequently, $\sqrt{M_{0}^{-1}}\left(A+M_{1}\right)\sqrt{M_{0}^{-1}}$
generates a unitary 1-parameter group and {}``energy'' conservation
holds in the sense that for the solution $U$ of a pure initial value
problem we have for $t\in\mathbb{R}_{>0}$ 
\[
\left|\sqrt{M_{0}}U\left(t\right)\right|_{H}=\left|\sqrt{M_{0}}U\left(0+\right)\right|_{H}
\]
or if one prefers to underscore the {}``energy'' metaphor
\[
E\left(t\right)\coloneq\frac{1}{2}\left|\sqrt{M_{0}}U\left(t\right)\right|_{H}^{2}=\frac{1}{2}\left|\sqrt{M_{0}}U\left(0+\right)\right|_{H}^{2}=E\left(0+\right).
\]
} 
\begin{align*}
M\left(\partial_{0}^{-1}\right) & =\left(\begin{array}{ccc}
\kappa_{0} & 0 & 0\\
0 & \kappa_{1} & -\mu_{0}^{*}\partial_{0}^{-1}\\
0 & \mu_{0}\partial_{0}^{-1} & \epsilon+\eta\,\partial_{0}^{-1}+\mu_{1}\partial_{0}^{-2}
\end{array}\right),
\end{align*}
where $\epsilon:L^{2}\left(I\right)\to L^{2}\left(I\right)$, $\kappa_{0}:L^{2}\left(I\right)\to L^{2}\left(I\right)$,
$\kappa_{1}:L^{2}\left(I\right)\to L^{2}\left(I\right)$ are suitable
continuous, selfadjoint, non-negative mappings and and $\mu_{0}:L^{2}\left(I\right)\to L^{2}\left(I\right)$,
$\mu_{1}:L^{2}\left(I\right)\to L^{2}\left(I\right)$, $\eta:L^{2}\left(I\right)\to L^{2}\left(I\right)$
are continuous and linear. We assume of course that the coefficient
operators are such that (\ref{eq:posdefmat}) is satisfied, i.e. 
\[
\nu\epsilon+\Re\eta\geq c_{0}>0
\]
for some $c_{0}\in\mathbb{R}$ and all sufficiently large $\nu\in\mathbb{R}_{>0}.$
Such material laws are suggested by models of linear acoustics, see
e.g. \cite{Leis:Buch:2}, or by the so-called Maxwell-Cattaneo-Vernotte
law \cite{Cattaneo.26203,Pi2009-1} describing heat propagation. In
the 1-dimensional case, focused on here, this special material law
operator can be reduced to a wave or heat equation type partial differential
operator with a Sturm-Liouville type operator as spatial part. Indeed,
assuming additionally that $\kappa_{0}$ and $\kappa_{1}$ are strictly
positive, two elementary row operations%
\footnote{A more common point of view for this operation would be to think of
new unknowns being introduced. Indeed, if the system unknowns are
$\left(\begin{array}{c}
s\\
w\\
v
\end{array}\right)$ then letting
\[
y\coloneq\partial_{0}^{-1}v
\]
 we would get from a line by line inspection of the system
\[
\left(\begin{array}{ccc}
\kappa_{0}\partial_{0} & 0 & \partial\\
0 & \kappa_{1}\partial_{0} & -\mu_{0}^{*}\\
\partial & \mu_{0} & \epsilon\partial_{0}+\eta+\mu_{1}\partial_{0}^{-1}
\end{array}\right)\left(\begin{array}{c}
s\\
w\\
v
\end{array}\right)=\left(\begin{array}{c}
0\\
0\\
f
\end{array}\right)
\]
that 
\begin{eqnarray*}
s & = & -\kappa_{0}^{-1}\partial y\\
w & = & \kappa_{1}^{-1}\mu_{0}^{*}y
\end{eqnarray*}
and so that
\[
\epsilon\partial_{0}^{2}y+\eta\partial_{0}y+\mu_{1}y+\mu_{0}\kappa_{1}^{-1}\mu_{0}^{*}y-\partial\kappa_{0}^{-1}\partial y=f.
\]
} applied to
\[
\left(\begin{array}{ccc}
\kappa_{0}\partial_{0} & 0 & \partial\\
0 & \kappa_{1}\partial_{0} & -\mu_{0}^{*}\\
\partial & \mu_{0} & \epsilon\partial_{0}+\eta+\mu_{1}\partial_{0}^{-1}
\end{array}\right)
\]
yield formally
\[
\left(\begin{array}{ccc}
\kappa_{0}\partial_{0} & 0 & \partial\\
0 & \kappa_{1}\partial_{0} & -\mu_{0}^{*}\\
0 & 0 & \partial_{0}^{-1}\left(\epsilon\partial_{0}^{2}+\eta\partial_{0}+\left(\mu_{0}\kappa_{1}^{-1}\mu_{0}+\mu_{1}\right)-\partial\kappa_{0}^{-1}\partial\right)
\end{array}\right).
\]
Clearly, applying the Fourier-Laplace transform to $\left(\epsilon\partial_{0}^{2}+\eta\partial_{0}+\left(\mu_{0}\kappa_{1}^{-1}\mu_{0}^{*}+\mu_{1}\right)-\partial\kappa_{0}^{-1}\partial\right)$
we obtain point-wise, writing $\sqrt{\lambda}$ instead of $\left(\i m+\nu\right)$,
\[
\left(\epsilon\lambda+\eta\sqrt{\lambda}+\left(\mu_{0}\kappa_{1}^{-1}\mu_{0}+\mu_{1}\right)-\partial\kappa_{0}^{-1}\partial\right),
\]
which for vanishing {}``damping'' $\eta$ is indeed a Sturm-Liouville
operator%
\footnote{If $\epsilon=0$ and $\eta$ has a strictly positive definite symmetric
part, i.e. the selfadjoint $\Re\eta$ is strictly positive, we arrive
at the parabolic type operator 
\[
\eta\partial_{0}+q-\partial p\partial
\]
and writing $\lambda$ instead of $\left(\i m+\nu\right)$ we get
again a Sturm-Liouville type operator
\[
\eta\lambda+q-\partial p\partial,
\]
where now $\eta$ plays the role of $r$.%
}: 
\[
r\lambda+q-\partial p\partial
\]

with
\begin{eqnarray*}
r & \coloneq & \epsilon,\\
q & \coloneq & \mu_{0}\kappa_{1}^{-1}\mu_{0}+\mu_{1}\,,\\
p & \coloneq & \kappa_{0}^{-1}.
\end{eqnarray*}
For our purposes we may allow for general material laws in the problem
(\ref{eq:Sturm-Louis}).

Denoting the inner product and norm of $H_{\nu,0}\left(\mathbb{R},L^{2}\left(I\right)\oplus L^{2}\left(I\right)\oplus L^{2}\left(I\right)\right)$
by $\left\langle \;\cdot\:|\;\cdot\:\right\rangle _{\nu,0,0}$ and
$\left|\;\cdot\:\right|_{\nu,0,0},$ respectively, we calculate
\begin{equation}
\begin{array}{l}
\Re\left\langle \chi_{_{]-\infty,0]}}\left(m_{0}\right)\left(\begin{array}{c}
s\\
w\\
v
\end{array}\right)\left|A\left(\begin{array}{c}
s\\
w\\
v
\end{array}\right)\right.\right\rangle _{\nu,0,0}=\\
=\Re\left(\left\langle \chi_{_{]-\infty,0]}}\left(m_{0}\right)s|\partial v\right\rangle _{\nu,0,0}+\left\langle \partial s|\chi_{_{]-\infty,0]}}\left(m_{0}\right)v\right\rangle _{\nu,0,0}\right)\\
=\Re\left\langle \chi_{_{]-\infty,0]}}\left(m_{0}\right)s|\interior{\partial}\left(v-\partial_{0}a\left(\partial_{0}^{-1}\right)s\right)\right\rangle _{\nu,0,0}+\\
+\Re\left\langle \chi_{_{]-\infty,0]}}\left(m_{0}\right)s|\partial\partial_{0}a\left(\partial_{0}^{-1}\right)s\right\rangle _{\nu,0,0}+\Re\left\langle \partial s|\chi_{_{]-\infty,0]}}\left(m_{0}\right)v\right\rangle _{\nu,0,0}\\
=-\Re\left\langle \partial s|\chi_{_{]-\infty,0]}}\left(m_{0}\right)\left(v-\partial_{0}a\left(\partial_{0}^{-1}\right)s\right)\right\rangle _{\nu,0,0}+\Re\left\langle \chi_{_{]-\infty,0]}}\left(m_{0}\right)s|\partial\partial_{0}a\left(\partial_{0}^{-1}\right)s\right\rangle _{\nu,0,0}+\\
+\Re\left\langle \partial s|\chi_{_{]-\infty,0]}}\left(m_{0}\right)v\right\rangle _{\nu,0,0}\\
=\Re\left(\left\langle \partial\chi_{_{]-\infty,0]}}\left(m_{0}\right)s|\partial_{0}a\left(\partial_{0}^{-1}\right)s\right\rangle _{\nu,0,0}+\left\langle \chi_{_{]-\infty,0]}}\left(m_{0}\right)s|\partial\partial_{0}a\left(\partial_{0}^{-1}\right)s\right\rangle _{\nu,0,0}\right)\\
=\Re\left\langle \chi_{_{]-\infty,0]}}\left(m_{0}\right)s\left(\,\cdot\:,+1/2\right)|\partial_{0}a\left(+1/2,\partial_{0}^{-1}\right)s\left(\,\cdot\:,+1/2\right)\right\rangle _{\nu,0}+\\
-\Re\left\langle \chi_{_{]-\infty,0]}}\left(m_{0}\right)s\left(\,\cdot\:,-1/2\right)|\partial_{0}a\left(-1/2,\partial_{0}^{-1}\right)s\left(\,\cdot\:,-1/2\right)\right\rangle _{\nu,0}.
\end{array}\label{eq:causalposdef-}
\end{equation}
 For this to be non-negative we assume that $a$ is such that
\begin{equation}
\Re\left\langle \chi_{_{]-\infty,0]}}\left(m_{0}\right)\varphi|\pm\partial_{0}a\left(\pm1/2,\partial_{0}^{-1}\right)\varphi\right\rangle _{\nu,0}\geq0\label{eq:positivity}
\end{equation}
for every $\varphi\in H_{\nu,1}\left(\mathbb{R}\right)$. Due to its
analyticity at $0$ the operators $a\left(x,\partial_{0}^{-1}\right)$
are of the form 
\[
a\left(\partial_{0}^{-1}\right)=a_{0}+a_{1}\partial_{0}^{-1}+a_{2}\partial_{0}^{-2}+\partial_{0}^{-3}a^{\left(3\right)}\left(\partial_{0}^{-1}\right),
\]
where $a^{\left(3\right)}\left(\partial_{0}^{-1}\right)$ is bounded.
With this we can analyse (\ref{eq:positivity}) further. It is
\begin{eqnarray*}
 &  & \Re\left\langle \chi_{_{]-\infty,0]}}\left(m_{0}\right)\varphi|\pm\partial_{0}a_{0}\left(\pm1/2\right)\varphi\right\rangle _{\nu,0}=\\
 &  & =\pm\int_{-\infty}^{0}\varphi\left(t\right)^{*}\left(\partial_{0}a_{0}\left(\pm1/2\right)\varphi\right)\left(t\right)\:\exp\left(-2\nu t\right)\: dt\\
 &  & =\pm\nu\int_{-\infty}^{0}a_{0}\left(\pm1/2\right)\left|\varphi\left(t\right)\right|^{2}\:\exp\left(-2\nu t\right)\: dt\pm\frac{1}{2}a_{0}\left(\pm1/2\right)\left|\varphi\left(0\right)\right|^{2}
\end{eqnarray*}
which is non-negative if we assume 
\begin{equation}
\pm a_{0}\left(\pm1/2\right)\geq0.\label{eq:aplusminus}
\end{equation}

Similarly
\begin{eqnarray*}
 &  & \Re\left\langle \chi_{_{]-\infty,0]}}\left(m_{0}\right)\varphi|\pm a_{1}\left(\pm1/2\right)\varphi\right\rangle _{\nu,0}=\\
 &  & =\pm\int_{-\infty}^{0}\varphi\left(t\right)^{*}a_{1}\left(\pm1/2\right)\varphi\left(t\right)\:\exp\left(-2\nu t\right)\: dt.
\end{eqnarray*}
Assuming that%
\footnote{If $a\left(\partial_{0}^{-1}\right)=a_{0}+a_{1}\partial_{0}^{-1}$
it is sufficient to require 
\[
\pm\nu a_{0}\left(\pm1/2\right)\pm a_{1}\left(\pm1/2\right)\geq0
\]
for all sufficiently large $\nu\in\mathbb{R}_{>0}.$%
}
\[
\pm\nu a_{0}\left(\pm1/2\right)\pm a_{1}\left(\pm1/2\right)\geq c_{0}>0
\]
for $\nu\in\mathbb{R}_{>0}$ sufficiently large we obtain
\begin{eqnarray*}
 &  & \Re\left\langle \chi_{_{]-\infty,0]}}\left(m_{0}\right)\varphi|\pm\partial_{0}a\left(\pm1/2,\partial_{0}^{-1}\right)\varphi\right\rangle _{\nu,0}\geq\\
 &  & \geq\left(\pm\nu a_{0}\left(\pm1/2\right)\pm a_{1}\left(\pm1/2\right)\right)\left|\chi_{_{]-\infty,0]}}\left(m_{0}\right)\varphi\right|_{\nu,0}^{2}+\\
 &  & +\Re\left\langle \chi_{_{]-\infty,0]}}\left(m_{0}\right)\varphi|\pm\partial_{0}^{-1}a^{\left(2\right)}\left(\pm1/2,\partial_{0}^{-1}\right)\varphi\right\rangle _{\nu,0}.
\end{eqnarray*}
Due to causality we have
\begin{eqnarray*}
 &  & \left|\Re\left\langle \chi_{_{]-\infty,0]}}\left(m_{0}\right)\varphi|\pm\partial_{0}^{-1}a^{\left(2\right)}\left(\pm1/2,\partial_{0}^{-1}\right)\varphi\right\rangle _{\nu,0}\right|=\\
 &  & =\left|\Re\left\langle \chi_{_{]-\infty,0]}}\left(m_{0}\right)\varphi|\pm a^{\left(2\right)}\left(\pm1/2,\partial_{0}^{-1}\right)\chi_{_{]-\infty,0]}}\left(m_{0}\right)\partial_{0}^{-1}\varphi\right\rangle _{\nu,0}\right|\\
 &  & \leq C_{1}\left|\chi_{_{]-\infty,0]}}\left(m_{0}\right)\varphi\right|_{\nu,0}\left|\chi_{_{]-\infty,0]}}\left(m_{0}\right)\partial_{0}^{-1}\varphi\right|_{\nu,0}\\
 &  & \leq C_{1}\left|\chi_{_{]-\infty,0]}}\left(m_{0}\right)\varphi\right|_{\nu,0}\left|\chi_{_{]-\infty,0]}}\left(m_{0}\right)\partial_{0}^{-1}\chi_{_{]-\infty,0]}}\left(m_{0}\right)\varphi\right|_{\nu,0}\\
 &  & \leq C_{1}\left|\chi_{_{]-\infty,0]}}\left(m_{0}\right)\varphi\right|_{\nu,0}\left|\partial_{0}^{-1}\chi_{_{]-\infty,0]}}\left(m_{0}\right)\varphi\right|_{\nu,0}\\
 &  & \leq C_{1}\nu^{-1}\left|\chi_{_{]-\infty,0]}}\left(m_{0}\right)\varphi\right|_{\nu,0}^{2}.
\end{eqnarray*}
Under this assumption we have
\begin{equation}
\Re\left\langle \chi_{_{]-\infty,0]}}\left(m_{0}\right)U|AU\right\rangle _{\nu,0,0}\geq0\label{eq:causalposdef}
\end{equation}
for all $U\in D\left(A\right)$ if $\nu\in\mathbb{R}_{>0}$ is sufficiently
large. Note that by the time-translation invariance this is the same
as saying
\begin{equation}
\Re\left\langle \chi_{_{]-\infty,a]}}\left(m_{0}\right)U|AU\right\rangle _{\nu,0,0}\geq0\label{eq:causalposdef+}
\end{equation}

for all $U\in D\left(A\right)$ and all $a\in\mathbb{R}$. Letting
$a\to\infty$ we obtain from this
\begin{equation}
\Re\left\langle U|AU\right\rangle _{\nu,0,0}\geq0\label{eq:posdefa}
\end{equation}
for all $U\in D\left(A\right)$. 

We need to find the adjoint of $A$. It must satisfy 

\[
-\left(\begin{array}{ccc}
0 & 0 & \interior{\partial}\\
0 & 0 & 0\\
\interior{\partial} & 0 & 0
\end{array}\right)\subseteq A^{*}\subseteq-\left(\begin{array}{ccc}
0 & 0 & \partial\\
0 & 0 & 0\\
\partial & 0 & 0
\end{array}\right)
\]
in the sense of extensions. We now show that $D\left(A^{*}\right)$
is given by{\small{} 
\[
\left\{ \left(\begin{array}{c}
s\\
w\\
v
\end{array}\right)\in H_{\nu,0}\left(\mathbb{R},H\left(\partial,I\right)\oplus L^{2}\left(I\right)\oplus H\left(\partial,I\right)\right)\:\Big|\; a\left(\left(\partial_{0}^{-1}\right)^{*}\right)s+\left(\partial_{0}^{-1}\right)^{*}v\in H_{\nu,0}(\mathbb{R},H\left(\interior{\partial},I\right)\right\} .
\]
}Indeed, for 
\[
\left(\begin{array}{c}
s\\
w\\
v
\end{array}\right)\in D\left(A\right)
\]
we have
\[
\left(\begin{array}{cc}
1 & 0\\
-a\left(\partial_{0}^{-1}\right) & \partial_{0}^{-1}
\end{array}\right)\left(\begin{array}{c}
s\\
v
\end{array}\right)\in H_{\nu,0}\left(\mathbb{R},H\left({\partial},I\right)\oplus H\left(\interior{\partial},I\right)\right).
\]
Direct computation gives
\begin{eqnarray*}
\left(\begin{array}{cc}
0 & \interior{\partial}\\
\partial & 0
\end{array}\right)\left(\begin{array}{cc}
1 & 0\\
-a\left(\partial_{0}^{-1}\right) & \partial_{0}^{-1}
\end{array}\right) & = & \left(\begin{array}{cc}
0 & \partial\\
\partial & 0
\end{array}\right)\left(\begin{array}{cc}
1 & 0\\
0 & \partial_{0}^{-1}
\end{array}\right)+\left(\begin{array}{cc}
0 & \partial\\
0 & 0
\end{array}\right)\left(\begin{array}{cc}
0 & 0\\
-a\left(\partial_{0}^{-1}\right) & 0
\end{array}\right)\\
 & = & \left(\begin{array}{cc}
\partial_{0}^{-1} & 0\\
0 & 1
\end{array}\right)\left(\begin{array}{cc}
0 & \partial\\
\partial & 0
\end{array}\right)-\left(\begin{array}{cc}
a^{\prime}\left(\partial_{0}^{-1}\right) & 0\\
0 & 0
\end{array}\right)+\\
 &  & -\left(\begin{array}{cc}
a\left(\partial_{0}^{-1}\right)\partial & 0\\
0 & 0
\end{array}\right)\\
 & = & \left(\begin{array}{cc}
\partial_{0}^{-1} & -a\left(\partial_{0}^{-1}\right)\\
0 & 1
\end{array}\right)\left(\begin{array}{cc}
0 & \partial\\
\partial & 0
\end{array}\right)-\left(\begin{array}{cc}
a^{\prime}\left(\partial_{0}^{-1}\right) & 0\\
0 & 0
\end{array}\right).
\end{eqnarray*}
Thus, we have 
\begin{eqnarray*}
\left(\begin{array}{cc}
0 & \partial\\
\partial & 0
\end{array}\right)\partial_{0}^{-1}\left(\begin{array}{c}
s\\
v
\end{array}\right) & = & \left(\begin{array}{cc}
1 & a\left(\partial_{0}^{-1}\right)\\
0 & \partial_{0}^{-1}
\end{array}\right)\left(\begin{array}{cc}
0 & \interior{\partial}\\
\partial & 0
\end{array}\right)\left(\begin{array}{cc}
1 & 0\\
-a\left(\partial_{0}^{-1}\right) & \partial_{0}^{-1}
\end{array}\right)\left(\begin{array}{c}
s\\
v
\end{array}\right)+\\
 &  & +\left(\begin{array}{cc}
1 & a\left(\partial_{0}^{-1}\right)\\
0 & \partial_{0}^{-1}
\end{array}\right)\left(\begin{array}{cc}
a^{\prime}\left(\partial_{0}^{-1}\right) & 0\\
0 & 0
\end{array}\right)\left(\begin{array}{c}
s\\
v
\end{array}\right)\\
 & = & \left(\begin{array}{cc}
1 & a\left(\partial_{0}^{-1}\right)\\
0 & \partial_{0}^{-1}
\end{array}\right)\left(\begin{array}{cc}
0 & \interior{\partial}\\
\partial & 0
\end{array}\right)\left(\begin{array}{cc}
1 & 0\\
-a\left(\partial_{0}^{-1}\right) & \partial_{0}^{-1}
\end{array}\right)\left(\begin{array}{c}
s\\
v
\end{array}\right)+\\
 &  & +\left(\begin{array}{cc}
a^{\prime}\left(\partial_{0}^{-1}\right) & 0\\
0 & 0
\end{array}\right)\left(\begin{array}{c}
s\\
v
\end{array}\right).
\end{eqnarray*}
Letting $\left(\begin{array}{cc}
1 & 0\\
-a\left(\partial_{0}^{-1}\right) & \partial_{0}^{-1}
\end{array}\right)\left(\begin{array}{c}
s\\
v
\end{array}\right)=W$ we have for $\left(\begin{array}{c}
v_{0}\\
v_{1}\\
v_{2}
\end{array}\right)\in D\left(A^{*}\right)$ and for every $W\in H_{\nu,0}\left(\mathbb{R},H\left({\partial},I\right)\oplus H\left(\interior{\partial},I\right)\right),$
\begin{eqnarray*}
0 & = & \left\langle \left(\begin{array}{cc}
0 & \partial\\
\partial & 0
\end{array}\right)\left(\begin{array}{cc}
\partial_{0}^{-1} & 0\\
a\left(\partial_{0}^{-1}\right) & 1
\end{array}\right)W|\left(\begin{array}{c}
v_{0}\\
v_{2}
\end{array}\right)\right\rangle _{\nu,0,0}+\left\langle \left(\begin{array}{cc}
\partial_{0}^{-1} & 0\\
a\left(\partial_{0}^{-1}\right) & 1
\end{array}\right)W|\left(\begin{array}{cc}
0 & \partial\\
\partial & 0
\end{array}\right)\left(\begin{array}{c}
v_{0}\\
v_{2}
\end{array}\right)\right\rangle _{\nu,0,0}\\
 & = & \left\langle \left(\begin{array}{cc}
1 & a\left(\partial_{0}^{-1}\right)\\
0 & \partial_{0}^{-1}
\end{array}\right)\left(\begin{array}{cc}
0 & \interior{\partial}\\
\partial & 0
\end{array}\right)W|\left(\begin{array}{c}
v_{0}\\
v_{2}
\end{array}\right)\right\rangle _{\nu,0,0}+\\
 &  & +\left\langle \left(\begin{array}{cc}
a^{\prime}\left(\partial_{0}^{-1}\right) & 0\\
0 & 0
\end{array}\right)W|\,\left(\begin{array}{c}
v_{0}\\
v_{2}
\end{array}\right)\right\rangle _{\nu,0,0}+\left\langle \left(\begin{array}{cc}
\partial_{0}^{-1} & 0\\
a\left(\partial_{0}^{-1}\right) & 1
\end{array}\right)W|\left(\begin{array}{cc}
0 & \partial\\
\partial & 0
\end{array}\right)\left(\begin{array}{c}
v_{0}\\
v_{2}
\end{array}\right)\right\rangle _{\nu,0,0}\\
 & = & \left\langle \left(\begin{array}{cc}
0 & \interior{\partial}\\
\partial & 0
\end{array}\right)W|\left(\begin{array}{cc}
1 & 0\\
a\left(\left(\partial_{0}^{-1}\right)^{*}\right) & \left(\partial_{0}^{-1}\right)^{*}
\end{array}\right)\left(\begin{array}{c}
v_{0}\\
v_{2}
\end{array}\right)\right\rangle _{\nu,0,0}+\\
 &  & +\left\langle W|\,\left(\begin{array}{cc}
a^{\prime}\left(\left(\partial_{0}^{-1}\right)^{*}\right) & 0\\
0 & 0
\end{array}\right)\left(\begin{array}{c}
v_{0}\\
v_{2}
\end{array}\right)\right\rangle _{\nu,0,0}+\\
 &  & +\left\langle W|\left(\begin{array}{cc}
\left(\partial_{0}^{-1}\right)^{*} & a\left(\left(\partial_{0}^{-1}\right)^{*}\right)\\
0 & 1
\end{array}\right)\left(\begin{array}{cc}
0 & \partial\\
\partial & 0
\end{array}\right)\left(\begin{array}{c}
v_{0}\\
v_{2}
\end{array}\right)\right\rangle _{\nu,0,0}.
\end{eqnarray*}
This implies that
\[
\left(\begin{array}{cc}
1 & 0\\
a\left(\left(\partial_{0}^{-1}\right)^{*}\right) & \left(\partial_{0}^{-1}\right)^{*}
\end{array}\right)\left(\begin{array}{c}
v_{0}\\
v_{2}
\end{array}\right)\in H_{\nu,0}\left(\mathbb{R},H\left({\partial},I\right)\oplus H\left(\interior{\partial},I\right)\right),
\]
which is the above characterization of $D\left(A^{*}\right)$. Moreover,
\begin{eqnarray*}
\left(\begin{array}{cc}
0 & \interior{\partial}\\
\partial & 0
\end{array}\right)\left(\begin{array}{cc}
1 & 0\\
a\left(\left(\partial_{0}^{-1}\right)^{*}\right) & \left(\partial_{0}^{-1}\right)^{*}
\end{array}\right)\left(\begin{array}{c}
v_{0}\\
v_{2}
\end{array}\right) & = & \left(\begin{array}{cc}
a^{\prime}\left(\left(\partial_{0}^{-1}\right)^{*}\right) & 0\\
0 & 0
\end{array}\right)\left(\begin{array}{c}
v_{0}\\
v_{2}
\end{array}\right)+\\
 &  & +\left(\begin{array}{cc}
\left(\partial_{0}^{-1}\right)^{*} & a\left(\left(\partial_{0}^{-1}\right)^{*}\right)\\
0 & 1
\end{array}\right)\left(\begin{array}{cc}
0 & \partial\\
\partial & 0
\end{array}\right)\left(\begin{array}{c}
v_{0}\\
v_{2}
\end{array}\right).
\end{eqnarray*}

As a consequence of the similarity between $A$ and $A^{*}$ we find
by analogous reasoning that we have not only (\ref{eq:posdefa}) but
also, indeed more straight-forwardly,
\begin{equation}
\Re\left\langle V|A^{*}V\right\rangle _{\nu,0,0}\geq0\label{eq:A*nonneg}
\end{equation}
for all $V\in D\left(A^{*}\right)$. The calculation is similar to
(\ref{eq:causalposdef-}) but without the cut-off with $\chi_{_{]-\infty,0}}\left(m_{0}\right)$.
Thus we have indeed that
\[
\partial_{0}M\left(\partial_{0}^{-1}\right)+A
\]
is continuously invertible with causal inverse $\left(\partial_{0}M\left(\partial_{0}^{-1}\right)+A\right)^{-1}:H_{\nu,k}\left(\mathbb{R},H\right)\to H_{\nu,k}\left(\mathbb{R},H\right)$
for every $k\in\mathbb{Z}$ and any sufficiently large $\nu\in\mathbb{R}_{>0}$.
We summarize our findings in the following theorem.
\begin{thm}
Under assumptions (\ref{eq:posdefmat}) and (\ref{eq:positivity})
we have that for every $\mathcal{J}\in H_{\nu,k}\left(\mathbb{R},H\right)$
the problem (\ref{eq:Sturm-Louis}) has a unique solution $U\in H_{\nu,k}\left(\mathbb{R},H\right).$
The solution operator $\left(\partial_{0}M\left(\partial_{0}^{-1}\right)+A\right)^{-1}:H_{\nu,k}\left(\mathbb{R},H\right)\to H_{\nu,k}\left(\mathbb{R},H\right)$
is continuous and causal for every $k\in\mathbb{Z}$ and any sufficiently
large $\nu\in\mathbb{R}_{>0}$.\end{thm}
\begin{proof}
Under the stated constraints the assumptions of Theorem \ref{SolutionTheory}
are satisfied and the result follows. \end{proof}

\end{document}